\documentclass[10pt,oneside,reqno]{amsart}
\usepackage{amsthm}
\usepackage{amsfonts}
\usepackage{amsmath}
\usepackage{amssymb}
\usepackage{mathtools}

\newtheorem{thm}{Theorem}[section]
\newtheorem{cor}[thm]{Corollary}
\newtheorem{lemma}[thm]{Lemma}
\newtheorem{prop}[thm]{Proposition}
\theoremstyle{definition}
\newtheorem{defn}[thm]{Definition}
\theoremstyle{remark}
\newtheorem{rem}[thm]{Remark}
\numberwithin{equation}{section}

\newtheorem{conj}[thm]{Conjecture}
\newtheorem{todo}{ToDo}

\def\<{\prec}
\def\>{\succ}
\def\O{{\mathcal O}}
\def\bes{\begin{equation*}}
\def\be{\begin{equation}}
\def\ee{\end{equation}}
\def\ees{\end{equation*}}

\def\a{\alpha}
\def\b{\beta}

\allowdisplaybreaks
\def\NS{N\!S}

\def\bC{\mathbb{C}}
\def\bN{\mathbb{N}}
\def\bR{\mathbb{R}}

\def\cR{\mathcal{R}}

\def\arcsinh{\operatorname{arcsinh}}

\def\ba{{\bf a}}

\def\bg{{\bf g}}
\def\bh{{\bf h}}
\def\bd{{\bf d}}

\def\bs{{\bf s}}

\begin{document}

\title[Asymptotic expansions of stable, stabilizable and stabilized means]{Asymptotic expansions of stable, stabilizable and stabilized means with applications}
\author{Lenka Mihokovi\'c}
\email{lenka.mihokovic@fer.hr}
\address{Lenka Mihokovi\'c, University of Zagreb, Faculty of Electrical Engineering and Computing, Department of Applied Mathematics, Unska 3, 10000 Zagreb, Croatia}
\begin{abstract}
	In this paper we present a complete asymptotic expansion of
	a symmetric homogeneous stable (balanced), stabilizable and stabilized mean.
	By including known asymptotic expansions of parametric means
it is shown how the obtained coefficients are used to solve the problem of identifying stable means within classes of parametric means under consideration,  how to disprove some mean is stabilizable or stabilized and how to obtain best possible parameters such that given mean is stabilizable with some pair of parametric means.
\end{abstract}
\maketitle

\begin{section}{Introduction}

Consider \emph{bi-variate mean} $M$, i.e.\ 
function $M\colon \bR^+ \times \bR^+\to\bR^+$
such that
\be\label{def-mean-minmax}
	 \min(s,t) \le M(s,t) \le \max(s,t).
\ee
We say that mean $M$ is symmetric if
\bes
	 M(s,t) =M(t,s),\ s,t>0,
\ees
and homogeneous (of degree 1) if
\bes
	 M(\lambda s,\lambda t)=\lambda M(s,t),\ \lambda,s,t,>0.
\ees
Let $M$, $N$, $K$, be three homogeneous symmetric bi-variate means and let
\bes
	\cR(K,N,M)(s,t)=K\Bigl( N\bigl(s,M(s,t)\bigr),N\bigl(M(s,t),t\bigr) \Bigr).
\ees
$\cR$ is also called the resultant mean--map of $K$, $M$ and $N$.
Observe the following functional equation
\be\label{def-stable}
	M(s,t) =\cR(M,M,M)(s,t)=M\Bigl( M\bigl(s,M(s,t)\bigr),M\bigl(M(s,t),t\bigr) \Bigr),
\ee
which has been examined by many authors in various settings. 

G.\ Aumann (\cite{Aumann-1934}) studied constructions of means of several arguments and corresponding iterative algorithms, more precisely augmentation of mean to $n+1$ arguments with mean of $n$ arguments given. He called such mean an {upper mean}. 
While studying the opposite procedure, i.e.\ the reduction process (\cite{Aumann-1935}), he introduced the lower mean. Wondering when these two processes are inverse to each other, for $n=2$ \emph{composite functional equation} \eqref{def-stable} appeared. 
G.\ Aumann proved that this functional equation in class of analytic means on $\bC^2$ 
is characteristic to the analytic quasi-arithmetic means.
L.\ R.\ Berrone (\cite{Berrone-2015}) presented key results from Aumann's two papers pointing out non-equivalence of complex methods within class of real variable means
and also analyzed generalizations of Aumann functional equation which involves general weighting operators.
%
%
T.\ Kiss (\cite{Kiss-2021}), calling the equation \eqref{def-stable} \emph{balancing property}, solved it without differentiability assumptions in a class of two-variable means, which contains the class of Matkowski means.

%
In this paper we follow definitions introduced by M.\ Ra\"{i}ssouli who, relying on equation \eqref{def-stable}, also introduced a notion of {stabilizable}  and {stabilized} mean. 
\begin{defn}[\cite{Raiss-2011-stab}]
	A symmetric mean $M$ is said to be:
	\begin{enumerate} 
		\item \emph{Stable} (\emph{balanced}), if $\cR(M,M,M)=M$.
		\item \emph{$(K,N)$-stabilizable}, if for two nontrivial stable means $K$ and $N$ following relation is satisfied:
			\be\label{def-stabilizable}
			M(s,t) =\cR(K,M,N)(s,t) = K\Bigl( M\bigl(s,N(s,t)\bigr),M\bigl(N(s,t),t\bigr) \Bigr).
			\ee
		\item \emph{$(K,N)$-stabilized}, if for two nontrivial stable means $K$ and $N$
		 following relation is satisfied:
		 	\be\label{def-stabilized}
				M(s,t) =\cR(K,N,M)(s,t) =K\Bigl( N\bigl(s,M(s,t)\bigr),N\bigl(M(s,t),t\bigr) \Bigr).
		\ee
	\end{enumerate}
\end{defn}

It can easily be seen that the arithmetic $A$ mean is stable and more general, the (binomial) power mean $B_p$ is also stable. The logarithmic and identric means are known not to be stable. 
Stabilizable or stabilized mean doesn't need to be stable by itself.
Furthermore, geometric mean $G$ is simultaneously $(A,H)$-stabilized and $(H,A)$-stabilized, 
while the Heron mean $H\!e$ is $(A,G)$-stabilized. 
The power binomial mean $B_p$, is stable for all real numbers $p$, 
the power logarithmic mean $L_p$, also known as generalized logarithmic mean, is $(B_p,A)$-stabilizable, 
the power difference mean $D_p$, i.e.\ Stolarsky mean $E_{p,p+1}$, is $(A,B_p)$-stabilizable, 
the power exponential mean $I_p$, i.e.\ Stolarsky mean $E_{p,p}$, is $(G,B_p)$-stabilizable 
and the second power logarithmic mean $l_p$, i.e.\ Stolarsky mean $E_{p,0}$, is $(B_p,G)$-stabilizable (\cite{Raiss-2011-stab}).
For all real numbers $p$ and $q$, Stolarsky mean $E_{p,q}$ is $(B_{q-p},B_p)$-stabilizable (\cite{Raiss-2012-stabStolarsky}). Precise definitions of those means will be given in Section 5. 

A given mean can be stabilizable with respect to two distinct couples of means. 
For instance, logarithmic mean $L$ is simultaneously $(A,G)$-stabilizable and $(H,A)$-stabilizable. 
On the other side, for two given stable means $M_1$ and $M_2$, such that $M_1\le M_2$ and $M_1$
is strict and a cross mean, there exists one and only one $(M_1,M_2)$-stabilizable mean $M$ such that $M_1\le M\le M_2$ (\cite{Raiss-2013-stabAnswer}). 

There are various applications of the stability and stabilizability.
As an extension of the stabilizability concept, A.\ Gasmi and M.\ Ra\"{i}ssouli (\cite{GasmiRaiss-2013-stabGeneralized}) introduced the generalized stabilizability for bi-variate means.
Regarding mean inequalities which have been studied extensively, M.\ Ra\"{i}ssouli (\cite{Raiss-2012-stabRefinements}) presented an approach for obtaining refinements in a convenient manner. He has shown how to obtain in a recursive way an infinite number of lower upper bounds starting from an arbitrary lower and upper bounds of a stabilizable mean.

For two nontrivial stable comparable means the (strict) sub-stability and super-stability concept can be introduced with the appropriate inequality sign in \eqref{def-stabilizable}.
 
 \begin{defn}[\cite{RaissSandor-2014}]
 	Let $K$, $N$ be two nontrivial stable comparable means. Mean $M$ is called
 	\begin{enumerate}
 		\item \emph{$(K,N)$-sub-stabilizable}, if $\cR(K,M,N)\le M$ and $K\le M\le N$,
 		\item \emph{$(K,N)$-super-stabilizable}, if $M\le\cR(K,M,N)$ and $K\le M\le N$.
 	\end{enumerate}
 \end{defn}
 
For example, geometric mean $G$ is $(G,A)$-super-stabilizable (but not strictly), arithmetic mean $A$ is $(G,A)$-sub-stabilizable, logarithmic mean $L$ is strictly $(G,A)$-super-stabilizable and strictly $(A,H)$-sub-stabilizable, identric mean $I$ is strictly $(A,G)$-sub-stabilizable (\cite{RaissSandor-2014}). The first Seiffert mean is strictly $(G,A)$-super-stabilizable (\cite{AnisiuAnisiu-2014}).

In the above mentioned papers some open problems appeared from which we shall single out the following.
\begin{enumerate}
\item Find all pairs $(p,q)$ such that Gini means $G_{p,q}$ and Stolarsky means $E_{p,q}$ are stable (\cite{Raiss-2011-stab}).
\item 
Prove or disprove that the first Seiffert mean $P$ is not stabilizable  (\cite{Raiss-2011-stab}).
\item Find the best real numbers $p>0$ and $q>0$ for which the first Seiffert means $P$ is strictly $(B_p,B_q)$-sub-stabilizable  (\cite{RaissSandor-2014}).
\item Are the second Seiffert mean $T$ and the Neuman-S\'andor mean $\NS$ strictly $(B_p,B_q)$-sub-stabilizable for some real numbers $p>0$, $q>0$  (\cite{RaissSandor-2014})?
\end{enumerate}

The aim of this paper is to apply the previously developed techniques of asymptotic expansions on the equations \eqref{def-stable}, \eqref{def-stabilizable} and \eqref{def-stabilized} in order to obtain the asymptotic expansion of stable, stabilizable and stabilized mean. 

		The (formal) series $\sum_{n=0}^{\infty} a_n \varphi_n(x)$ is said to be an \emph{asymptotic expansion} of a function $f(x)$ as $x\to x_0$, 
		with respect to asymptotic sequence $(\varphi_n(x))_{n\in \bN_0}$, 
		if for each $N\in\bN_0$
		$$
			f(x) \sim \sum_{n=0}^{N}a_n \varphi_n(x) +o(\varphi_N(x)).
		$$
	Approximation of the function to a given accuracy is achieved by approaching the variable to a certain fixed point or a point at infinity. A small number of members of this series ensures a good approximation. Taylor series can also be seen as an asymptotic as $x\to0$. Asymptotic series may be convergent or divergent. For a given asymptotic sequence, asymptotic representation is unique. Conversely, asymptotic series represents a class of asymptotically equal functions.
	Theoretical background from theory of asymptotic expansions can be found in \cite{Erd}.

For a bi-variate symmetric homogeneous stable mean
we find coefficients $a_n$ in the asymptotic power series expansion of the form
\be\label{mean-asymexp}
	M(x-t,x+t) \sim  \sum_{n=0}^\infty a_n t^{2n}x^{-2n+1},\ \text{as } x\to\infty,
\ee
which will be given by a recursive relation. 
It will be shown that the asymptotic representation \eqref{mean-asymexp} is sufficient to obtain the general form
\be\label{mean-asymexp-st}
	M(x+s,x+t) \sim \sum_{n=0}^\infty a_n(s,t) x^{-n+1},\ \text{as } x\to\infty,
\ee
where $a_n \equiv a_{2n}(-t,t)$. 


	Based on the asymptotic expansions, more precisely on positivity of the first non-zero coefficient, there can be introduced notion of asymptotic inequality. Recall its definition.
	\begin{defn}[\cite{Vu}]
	Let $F(s,t)$ be any homogeneous bi-variate function such that
	\bes
		F(x+s,x+t)
		=c_{k}(t,s)x^{-k+1}+\O(x^{-k}).
	\ees
	If $c_{k}(s,t)>0$ for all $s$ and $t$, then
	we say $F$ is \textit{asymptotically} greater than zero, and write
	\bes
		F\succ 0.
	\ees
\end{defn}
Asymptotic inequality is considered as a necessary relation between comparable means.
Namely, if $F\ge0$, then $F\succ 0$, which has been proved in the same paper. Furthermore, for the asymptotic inequalities it is sufficient to observe the case $s=-t$ as explained in \cite{El}. Asymptotic inequalities were used to obtain the best possible parameters in convex combinations of means which include Seiffert (\cite{Vu}) and Neuman-S\'andor (\cite{ElVu-2014-20}) mean and to obtain the best possible parameters such that inequality between some parametric means hold (\cite{ElVu-2014-04}).
In this paper asymptotic inequalities will be used to treat the case of sub-stabilizability with power means.

This paper is organized as follows. In Section 2 
we state some fundamental results regarding operations with asymptotic expansions and show the auxiliary result which will be used afterwards. In Section 3 
we obtain asymptotic expansion of the resultant mean-map provided that all three means involved possess asymptotic expansion. Using this result in order we obtained the asymptotic expansion of stable, stabilizable and stabilized mean. In Section 4 
we obtain necessary conditions for mean $N$ to be simultaneously $(K,M)$ and $(M,K)$-stabilizable, for mean $M$ to be simultaneously $(K,N)$ and $(N,K)$-stabilized and for mean $M$ to be simultaneously $(K,N)$-stabilizable and $(K,N)$-stabilized.
With respect to known asymptotic expansions of parametric means derived in \cite{ElVu-2014-04} and \cite{Vu}, in Section 5 
it will be shown how the obtained coefficients are used to solve the problem of identifying stable means within classes of parametric means under consideration, how to disprove some mean is stabilizable or stabilized and how to obtain best possible parameters such that given mean is stabilizable with some pair of parametric means. Recursive formulas were evaluated using computer algebra system \emph{Mathematica}.
In Section 6 
we sum up all the results, emphasize our conribution to the open questions from cited papers and state new conjectures which arose form this paper.
\end{section}

\begin{section}{Preliminaries}\label{section-Prelim}

Suppose that all means involved here have the asymptotic expansions as $x\to\infty$ of the following type
\begin{align}
	M(x-t,x+t) &\sim \sum_{n=0}^\infty a^M_n t^{2n}x^{-2n+1},\label{def-MNK-M}\\
	N(x-t,x+t) &\sim \sum_{n=0}^\infty a^N_n t^{2n}x^{-2n+1},\label{def-MNK-N}\\
	K(x-t,x+t) &\sim \sum_{n=0}^\infty a^K_n t^{2n}x^{-2n+1}.\label{def-MNK-K}
\end{align}

Operations with asymptotic power series are conducted in very intuitive manner.
Asymptotic expansion of a linear combination corresponds to expansion with same linear combination done term-wise. Coefficients in product of two asymptotic power series are defined by convolution. Also, two asymptotic power series can be divided with the result given in a form of asymptotic series as described in \cite[Lemma 1.1.]{ElVu-2014-04}. 
The composition has asymptotic expansion whose coefficients can be obtained by formal substitution and rearrangement of terms (\cite[p. 20]{Erd}).
Under some reasonable assumptions, asymptotic power series can be differentiated and integrated term by term (\cite[p. 21]{Erd}).
	In the sequel we state the fundamental result on transformations which is about power of an asymptotic series. Coefficients of the new series, which depend on the power $r$ and initial sequence $\ba=(a_n)_{n\in\bN_0}=(a_0,a_1,a_2,\ldots)$, will be denoted here as $P[n,r,\ba]$. We assume all sequences are enumerated from 0.

\begin{lemma}[\cite{ChenElVu-2013,Gould-1974}]
	\label{lemma-power}
	Let 
	$$
		g(x)\sim \sum_{n=0}^\infty a_n x^{-n}
	$$
	be a given asymptotic expansion (for $x\to\infty$) of $g(x)$ with $a_0\neq0$. Then for all real $r$ it holds
	$$
		[g(x)]^{r}\sim \sum_{n=0}^\infty P[n,r,\ba]x^{-n},
	$$
	where $P[0,r,\ba]=a_0^r$ and
	\bes
	        P[n,r,\ba]=\frac1{na_0}\sum_{k=1}^n[k(1+r)-n]a_kP[n-k,r,\ba],\quad n\in\bN.
	\ees
\end{lemma}

\begin{rem}\label{remark-powerlemma}
	It may be useful to consider $P[n,r,\ba]$ as the coefficient by the $x^{-n}$ in the $r$-th power of series assigned to sequence $\ba$, especially when $r$ is a nonnegative integer, wherefrom following useful relations easily follow.
	\begin{enumerate}
		\item $P[n,0,\ba] =\delta_n$, $n\in\bN_0$,  where $\delta_n$ stands for a single-argument Kronecker delta function.
		\item $P[n,1,\ba] = a_n$, $n\in\bN_0$.
		\item $P[0,r,\ba] = a_0^r$, $r\in\bR$.
	\end{enumerate}
\end{rem}

The following auxiliary sequences will be used to express main results. Let
\be\label{def-bgbh1}
\begin{aligned}
	\bg&\coloneqq(1,a_1,0,a_2,0,a_3,\ldots)\\
	\bh&\coloneqq(2,-1,a_1,0,a_2,0,a_3,\ldots).
\end{aligned}
\ee
and also
\bes
\begin{aligned}
	\tilde{\bg}&\coloneqq(1,-a_1,0,-a_2,0,-a_3,\ldots)\\
	\tilde{\bh}&\coloneqq(2,1,a_1,0,a_2,0,a_3,\ldots).
\end{aligned}
\ees

Let us denote by $D(m,n,k)$ and $S(m,n,k)$ terms which will appear within the inner sums later in proof of Theorem \ref{thm-main-RMNK}:
\begin{align}
	D(m,n,k) &= P[k,2n,\tilde{\bg}] P[m-2n-k,-2n+1,\tilde{\bh}]\label{def-Dmnk}\\
		&\quad-P[k,2n,\bg]P[m-2n-k,-2n+1,\bh],\notag \\
	S(m,n,k) &= P[k,2n,\tilde{\bg}] P[m-2n-k,-2n+1,\tilde{\bh}] \label{def-Smnk}\\
		&\quad+P[k,2n,\bg]P[m-2n-k,-2n+1,\bh].\notag
\end{align}
Some of the coefficients $D(m,n,k)$ and $S(m,n,k)$ are equal to zero because of the relation between sequences $\bg$ and $\tilde{\bg}$ and also $\bh$ and $\tilde{\bh}$. 

\begin{lemma}\label{lemma-Dmnk-reduced}
For $m\in\bN_0$, $n\in\{0,1,\ldots,\lfloor \frac{m}2 \rfloor\}$, $k\in \{0,1,\ldots,m-2n\}$, it holds
\be\label{def-Dmnk-reduced}
 	D(m,n,k)=\begin{cases}
 			0,\ & m\text{ even},\\
 			-2P[k,2n,\bg]P[m-2n-k,-2n+1,\bh], &m\text{ odd},
 		\end{cases}
\ee
and
\be\label{def-Smnk-reduced}
 	S(m,n,k)=\begin{cases}
 			2P[k,2n,\bg]P[m-2n-k,-2n+1,\bh],\ & m\text{ even},\\
 			0,& m\text{ odd}.
 		\end{cases}
\ee
\end{lemma}
\begin{proof}
Let us define (generating) functions
\bes
	G(x) =\sum_{k=0}^\infty g_k x^{-k}, \qquad
	\tilde{G}(x) =\sum_{k=0}^\infty \tilde{g}_k x^{-k},
\ees
whose $r$-th power can be expressed as
\bes
	[G(x)]^r=\sum_{j=0}^\infty P[j,r,\bg]x^{-j},\qquad
	[\tilde{G}(x)]^r=\sum_{j=0}^\infty P[j,r,\tilde{\bg}]x^{-j}.
\ees
Connection between coefficients in the expansion of the power of functions $G$ and $\tilde{G}$ can be established using underlying series  $A_1$, the generating function of a sequence 
	$\overline{\ba}=(a_1,0,a_2,0,\ldots)$:
\bes
	A_1(x)= \sum_{k=0}^\infty a_{k+1} x^{-2k}.
\ees
It holds
\begin{align*}
	[G(x)]^r &= (1+x^{-1} A_1(x))^r 
		 =\sum_{k=0}^\infty \binom{r}{k}x^{-k}A_1(x)^{k}\\
		&=\sum_{k=0}^\infty \binom{r}{k}x^{-k}
			\sum_{l=0}^\infty P[l,k,\overline{\ba}]x^{-2l}
		=\sum_{j=0}^\infty \sum_{l=0}^{\lfloor\frac{j}2\rfloor}
			 \binom{r}{j-2l} P[l,j-2l,\overline{\ba}] x^{-j}
\end{align*}
and
\begin{align*}
	[\tilde{G}&(x)]^r = (1-x^{-1} A_1(x))^r 
		 =\sum_{k=0}^\infty \binom{r}{k}(-1)^k x^{-k}A_1(x)^{k}\\
		&=\sum_{k=0}^\infty \binom{r}{k}(-1)^k x^{-k}
			\sum_{l=0}^\infty P[l,k,\overline{\ba}]x^{-2l}
		 =\sum_{j=0}^\infty (-1)^j \sum_{l=0}^{\lfloor\frac{j}2\rfloor}
			\binom{r}{j-2l} P[l,j-2l,\overline{\ba}] x^{-j}
\end{align*}
wherefrom it follows that
\be\label{relation-Pjrg}
	P[j,r,\bg] = (-1)^j	P[j,r,\tilde{\bg}].
\ee
Furthermore, let
\begin{align*}
	H(x) =\sum_{k=0}^\infty h_k x^{-k},\qquad
	\tilde{H}(x) =\sum_{k=0}^\infty \tilde{h}_k x^{-k},
\end{align*}
and also
\bes
	[H(x)]^r =\sum_{j=0}^\infty P[j,r,\bh]x^{-j},\qquad
	[\tilde{H}(x)]^r=\sum_{j=0}^\infty P[j,r,\tilde{\bh}]x^{-j}.
\ees
If $A_2$ denotes the generating function of a sequence 
	$\tilde{\ba}=(2,0,a_1,0,a_2,\ldots)$:
\bes
	A_2(x)= 2 + \sum_{k=1}^\infty a_{k} x^{-2k},
\ees
then the $r$-th power of functions $H$ and $\tilde{H}$ can be written as
\begin{align*}
	[&H(x)]^r = (A_2(x)-x^{-1})^r 
		 = A_2(x)^r (1-x^{-1} A_2(x)^{-1})^r \\
		&=\sum_{k=0}^\infty \binom{r}{k} (-1)^k x^{-k}A_2(x)^{r-k} 
		 =\sum_{k=0}^\infty \binom{r}{k} (-1)^k x^{-k}
			\sum_{l=0}^\infty P[l,r-k,\tilde{\ba}]x^{-2l}\\
		&=\sum_{j=0}^\infty (-1)^j \sum_{l=0}^{\lfloor\frac{j}2\rfloor}
			\binom{r}{j-2l} P[l,r+2l-j,\tilde{\ba}] x^{-j}
\end{align*}
and similarly
\begin{align*}
	 [\tilde{H}(x)]^r &= A_2(x)^r (1+x^{-1} A_2(x)^{-1})^r 
		 =\sum_{k=0}^\infty \binom{r}{k} x^{-k}A_2(x)^{r-k}\\
		&=\sum_{k=0}^\infty \binom{r}{k} x^{-k}
			\sum_{l=0}^\infty P[l,r-k,\tilde{\ba}]x^{-2l}
		=\sum_{j=0}^\infty \sum_{l=0}^{\lfloor\frac{j}2\rfloor}
			\binom{r}{j-2l} P[l,r+2l-j,\tilde{\ba}] x^{-j}
\end{align*}
 wherefrom it follows that
\be\label{relation-Pjrh}
	P[j,r,\bh] = (-1)^j	P[j,r,\tilde{\bh}].
\ee
Combining relations \eqref{relation-Pjrg} and \eqref{relation-Pjrh} with the definitions of $D$ and $S$ (\eqref{def-Dmnk} and \eqref{def-Smnk}) gives the relations \eqref{def-Dmnk-reduced} and \eqref{def-Smnk-reduced}.
\end{proof}

\end{section}

\pagebreak
\begin{section}{Main results}\label{section-main}
\setcounter{subsection}{0}

In all three notions, stable, stabilizable and stabilized, the similar composition appears.
The following theorem establishes the asymptotic expansion of the resultant mean--map of $K$, $M$ and $N$:
\be\label{eq-RKMN}
	R(x-t,x+t)\coloneqq \cR(K,N,M)(x-t,x+t)\sim\sum_{m=0}^\infty a_m^R t^{2m} x^{-2m+1}.
\ee
Afterwards, this composition will be used with $R=K=N=M$ to obtain the asymptotic expansion of stable mean $M$, with $R=N$ to obtain the asymptotic expansion of $(K,M)$-stabilizable mean $N$, and with $R=M$ to obtain the asymptotic expansion of $(K,N)$-stabilized mean $M$.

\begin{thm}\label{thm-main-RMNK}
Let homogeneous symmetric means $M$, $N$ and $K$ have the asymptotic expansions \eqref{def-MNK-M}, \eqref{def-MNK-N} and \eqref{def-MNK-K}. Then the coefficients $(a_m^R)_{m\in\bN_0}$ in the asymptotic expansion \eqref{eq-RKMN} are given by the formula:
\begin{align}
	a_m^R &=\sum_{n=0}^{m} a_n^K \sum_{k=0}^{m-n}
			P[k,2n,\bd] P[m-n-k,-2n+1,\bs], \quad m\in\bN_0,\label{thm-R-formula-aR}
\end{align}
where $\bd=(d_m)_{m\in\bN_0}$, $\bs=(s_m)_{m\in\bN_0}$, with
\begin{align}
	d_m &= -\frac12 \sum_{n=0}^{m} a_n^N 
		\sum_{k=0}^{2m+1-2n} P[k,2n,\bg^M]P[2m+1-2n-k,-2n+1,\bh^M], \quad m\in\bN_0,\label{def-dm}\\
	s_m &= \frac12\sum_{n=0}^{m} a_n^N 	\sum_{k=0}^{2m-2n} 
		 P[k,2n,\bg^M]P[2m-2n-k,-2n+1,\bh^M], \quad m\in\bN_0, \label{def-sm}
\end{align}
and $\bg^M$ and $\bh^M$ are defined as in \eqref{def-bgbh1} where $a_m=a_m^M$.
\end{thm}

\begin{proof}
First, we shall start from the composition $N\big(x-t,M(x-t,x+t)\big)$. We write arguments $x-t$ and $M=M(x-t,x+t)$ in form of difference and sum of terms $\frac12(M-x+t)$ and $\frac12(M+x-t)$. Then we apply expansion \eqref{def-MNK-N}, substitute $M$ with its asymptotic expansion \eqref{def-MNK-M}, use Lemma \ref{lemma-power} and rearrange sums to obtain the following:
\begin{align}\label{thm-proof-Nleft}
	&N\big(x-t,M(x-t,x+t)\big) \\
		&=N\left(\tfrac12(M+x-t)-\tfrac12(M-x+t), \tfrac12(M+x-t)+\tfrac12(M-x+t)\right)\notag\\
		&\sim \sum_{n=0}^\infty a_n^N\left(\tfrac12(M-x+t)\right)^{2n}
			\left(\tfrac12(M+x-t)\right)^{-2n+1} \notag\\
		&\sim \frac12 \sum_{n=0}^\infty a_n^N  \biggl(1+ \sum_{k=1}^\infty a_k^M t^{2k-1}x^{-2k+1} \biggr)^{2n}
			 \biggl( 2-\frac{t}{x}+\sum_{j=1}^\infty a_j^M t^{2j}x^{-2j}\biggr)^{-2n+1} t^{2n} x^{-2n+1}\notag \\
		&\sim\frac12 \sum_{n=0}^\infty a_n^N 
			\sum_{k=0}^\infty P[k,2n,\bg^M]t^kx^{-k} 
			\sum_{j=0}^\infty P[j,-2n+1,\bh^M]t^jx^{-j} t^{2n} x^{-2n+1}\notag \\
		&\sim\frac12 \sum_{m=0}^\infty \sum_{n=0}^{\lfloor \frac{m}2 \rfloor} a_n^N 
			\sum_{k=0}^{m-2n} P[k,2n,\bg^M] P[m-2n-k,-2n+1,\bh^M] t^{m} x^{-m+1}. \notag
\end{align}
With similar procedure, we have
\begin{align}\label{thm-proof-Nright}
	&N\big(M(x-t,x+t),x+t\big) \sim \sum_{n=0}^\infty a_n^N\left(\tfrac12(x+t-M)\right)^{2n}
			\left(\tfrac12(x+t+M)\right)^{-2n+1}\\
		&\sim \frac12 \sum_{n=0}^\infty a_n^N \biggl(1- \sum_{k=1}^\infty a_k^M t^{2k-1}x^{-2k+1} \biggr)^{2n}
			 \biggl( 2+\frac{t}{x}+\sum_{j=1}^\infty a_j^M t^{2j}x^{-2j}\biggr)^{-2n+1} t^{2n} x^{-2n+1} \notag\\
		&\sim\frac12 \sum_{n=0}^\infty a_n ^N
			\sum_{k=0}^\infty P[k,2n,\tilde{\bg}^M]t^k x^{-k} 
			\sum_{j=0}^\infty P[j,-2n+1,\tilde{\bh}^M]t^j x^{-j} t^{2n} x^{-2n+1} \notag\\
		&\sim\frac12 \sum_{m=0}^\infty \sum_{n=0}^{\lfloor \frac{m}2 \rfloor} a_n^N
			\sum_{k=0}^{m-2n} P[k,2n,\tilde{\bg}^M] P[m-2n-k,-2n+1,\tilde{\bh}^M] t^{m} x^{-m+1}.\notag
\end{align}
Now the left hand side in \eqref{eq-RKMN} we may write as
\be\label{thm-proof-R--K--X-T--X+T}
	K(X-T,X+T)=\sum_{n=0}^\infty a_n^K T^{2n} X^{-2n+1},
\ee
where $X$ and $T$ are such that their difference equals $N(x-t,M)$ and their sum equals $N(M,x+t)$, with $M=M(x-t,x+t)$.
We may further analyze $T$ and $X$. With use of \eqref{thm-proof-Nleft}, \eqref{thm-proof-Nright}, \eqref{def-Dmnk}, \eqref{def-Dmnk-reduced} and \eqref{def-dm} we obtain the following
\begin{align} \label{thm-proof-exp-T-dm}
	T&=\frac12\left( N(M(x-t,x+t),x+t)-N(x-t,M(x-t,x+t)) \right)\\
		&\sim \frac14 \sum_{m=0}^\infty \sum_{n=0}^{\lfloor \frac{m}2 \rfloor} a_n^N
			\sum_{k=0}^{m-2n} \bigl( P[k,2n,\tilde{\bg}^M] P[m-2n-k,-2n+1,\tilde{\bh}^M] \notag\\
				&\qquad \qquad \qquad \qquad -P[k,2n,\bg^M]P[m-2n-k,-2n+1,\bh^M] \bigr)
			 t^{m} x^{-m+1} \notag\\ 
		&\sim\frac14 \sum_{m=0}^\infty \sum_{n=0}^{\lfloor \frac{m}2 \rfloor} a_n^N 	\sum_{k=0}^{m-2n} D(m,n,k) t^{m} x^{-m+1}\notag\\
		&\sim\frac14 \sum_{m=0}^\infty \sum_{n=0}^{m} a_n^N 	\sum_{k=0}^{2m+1-2n} D(2m+1,n,k) t^{2m+1} x^{-2m}\notag\\
		& \sim -\frac12 \sum_{m=0}^\infty \sum_{n=0}^{m} a_n^N 	
			\sum_{k=0}^{2m+1-2n} P[k,2n,\bg^M]P[2m+1-2n-k,-2n+1,\bh^M]t^{2m+1} x^{-2m}\notag\\
		&  \sim\sum_{m=0}^\infty d_m t^{2m+1} x^{-2m} \notag
\end{align}
and similarly, with use of \eqref{thm-proof-Nleft}, \eqref{thm-proof-Nright}, \eqref{def-Smnk}, \eqref{def-Smnk-reduced} and \eqref{def-sm} we obtain the following
\begin{align}\label{thm-proof-exp-X-sm}
	X&=\frac12\left( N(M(x-t,x+t),x+t)+N(x-t,M(x-t,x+t)) \right)\\
		&\sim\frac14 \sum_{m=0}^\infty \sum_{n=0}^{\lfloor \frac{m}2 \rfloor} a_n^N 
			\sum_{k=0}^{m-2n} \bigl( P[k,2n,\tilde{\bg}^M] P[m-2n-k,-2n+1,\tilde{\bh}^M] \notag \\
				&\qquad \qquad \qquad \qquad+P[k,2n,\bg^M]P[m-2n-k,-2n+1,\bh^M] \bigr)
			 t^{m} x^{-m+1} \notag\\ 
		&\sim\frac14 \sum_{m=0}^\infty \sum_{n=0}^{\lfloor \frac{m}2 \rfloor} a_n^N 	\sum_{k=0}^{m-2n} S(m,n,k)  t^{m} x^{-m+1}\notag\\
	 &\sim\frac14 \sum_{m=0}^\infty \sum_{n=0}^{m} a_n^N  	\sum_{k=0}^{2m-2n} S(2m,n,k)  t^{2m} x^{-2m+1}\notag\\
		&\sim \frac12 \sum_{m=0}^\infty \sum_{n=0}^{m} a_n^N  	\sum_{k=0}^{2m-2n}
			 P[k,2n,\bg^M]P[2m-2n-k,-2n+1,\bh^M]t^{2m} x^{-2m+1}\notag\\
		&\sim\sum_{m=0}^\infty s_m t^{2m} x^{-2m+1}. \notag
\end{align}

Finally, from \eqref{eq-RKMN}, by including expansions of $K$ \eqref{def-MNK-K}, $T$ \eqref{thm-proof-exp-T-dm} and $X$ \eqref{thm-proof-exp-X-sm} in \eqref{thm-proof-R--K--X-T--X+T},
then using Lemma \ref{lemma-power} and rearranging sums we obtain:
\begin{align}
	R&= K(X-T,X+T)\notag \\
	&\sim\sum_{n=0}^\infty a_n^K \biggl( \sum_{k=0}^\infty d_k t^{2k+1} x^{-2k}\biggr)^{2n}
		\biggl(  \sum_{j=0}^\infty s_j t^{2j} x^{-2j+1} \biggr)^{-2n+1}\notag\\
		&\sim\sum_{n=0}^\infty a_n^K \left( \sum_{k=0}^\infty d_k t^{2k} x^{-2k}\right)^{2n}
		\biggl(  \sum_{j=0}^\infty s_j t^{2j} x^{-2j} \biggr)^{-2n+1} t^{2n} x^{-2n+1} \notag\\
		&\sim\sum_{n=0}^\infty a_n^K  
			\sum_{k=0}^\infty P[k,2n,\bd] t^{2k} x^{-2k}
			\sum_{j=0}^\infty P[j,-2n+1,\bs] t^{2j} x^{-2j}
			 t^{2n} x^{-2n+1} \notag\\
		&\sim\sum_{m=0}^\infty \sum_{n=0}^{m} a_n^K \sum_{k=0}^{m-n}
			P[k,2n,\bd] P[m-n-k,-2n+1,\bs]t^{2m} x^{-2m+1}. \notag  
\end{align} 
\end{proof}

\begin{rem}
	Asymptotic expansion of the composite mean $K(N_1,N_2)$ has been derived 
	by Burić and Elezović (\cite{BurEl-2017}) for two types of asymptotic power series, general ($\sum_{n=0}^\infty \gamma_n t^n x^{-n+1}$) in Theorem 2.1.\ and symmetric ($\sum_{n=0}^\infty \gamma_n t^{2n} x^{-2n+1}$) in Theorem 2.2. In our case means $N_1(s,t)=N(s,M(s,t))$ and $N_2(s,t)=N(M(s,t),t)$ are not symmetric and would require non-symmetric treatment. But due to specificity of the means $N_1$ and $N_2$, their difference is antisymmetric, their sum is symmetric and in the end the composition $K(N_1,N_2)$ is symmetric for symmetric means $K$, $N$ and $M$. Symmetric form of the asymptotic expansion of $K(N_1,N_2)$ would be difficult to deduce just by applying Theorem 2.1.\ from the above mentioned paper so we needed to conduct the similar procedure from scratch in order to obtain the desired result.
\end{rem}

According to Theorem \ref{thm-main-RMNK}, first few coefficients $a^R_m$ are as follows:
\be\label{list-coeff-Thm2-aR}
\begin{aligned}
 a^R_0 &=1, \\
 a^R_1 &=\frac14 (a^K_1+2 a^M_1+a^N_1), \\ 
 a^R_2 &=\frac1{16} (a^K_2+8 a^M_2+a^N_1+2 a^M_1 (1+2 a^M_1) a^N_1 \\
 		&\qquad-a^K_1 (3 a^N_1+a^M_1 (2+8 a^N_1))+a^N_2), \\
 a^R_3 &=\frac1{64}(a^K_3+32 a^M_3+(1-2 a^M_1 (1+2 a^M_1)^2+8 a^M_2+32 a^M_1 a^M_2) a^N_1\\
 &\qquad-a^K_2 (7 a^N_1+2 a^M_1 (3+8 a^N_1))+a^K_1 (a^N_1 (-3+4 a^N_1)-8 a^M_2 (1+4 a^N_1) \\
 &\qquad+4 (a^M_1)^2 (1+a^N_1) (1+4 a^N_1)+2 a^M_1 (a^N_1 (3+8 a^N_1)-8 a^N_2)-7 a^N_2)\\
 &\qquad+6 a^N_2+6 a^M_1 (3+4 a^M_1) a^N_2+a^N_3).
\end{aligned}
\ee


\begin{subsection}{Stable means}
In order to obtain asymptotic expansion of stable mean $M$ we need to use Theorem \ref{thm-main-RMNK} with $R=N=K=M$. The idea is to express coefficient $a_m=a_m^M$ using lower terms, i.e.\ in form of recursive relation.

\begin{thm}\label{thm-main}
Let homogeneous symmetric stable mean $M$ have the asymptotic expansion \eqref{mean-asymexp} with $a_m=a^M_m$. Then $a_0=1$, $a_1\in\bR$ and for $m\ge2$ coefficients $a_m$ are given by the recursive formula:
\be\label{thm-main-stable-am}
\begin{aligned}
	a_m  &= \frac{2^{2m-1}}{2^{2m-2}-1} \biggl(
			 \frac12\sum_{n=1}^{m-1} a_n 	\sum_{k=0}^{2m-2n} 
			 	P[k,2n,\bg]P[2m-2n-k,-2n+1,\bh] \\
			 &\qquad+\sum_{n=1}^{m-1} a_n \sum_{k=0}^{m-n}P[k,2n,\bd] P[m-n-k,-2n+1,\bs]\biggr),
			 \quad m\ge2,
\end{aligned}
\ee
where $\bg$ and $\bh$ are defined in \eqref{def-bgbh1} and $\bd$ and $\bs$ are defined by \eqref{def-dm} and \eqref{def-sm}.

\end{thm}

\begin{proof}
Proof is based on definition relation \eqref{def-stable} with variables $x-t$ and $x+t$:
\be\label{def-stable-xt}
	M(x-t,x+t)=M\bigl( M(x-t,M(x-t,x+t)),M(M(x-t,x+t),x+t) \bigr).
\ee
Proof is divided into three parts. The asymptotic expansions of the left hand side has the form \eqref{mean-asymexp} while the coefficients in the asymptotic expansion of the right hand side will be obtained as a consequence Theorem \ref{thm-main-RMNK}. Then term $a_m$ with the highest index will be identified. The corresponding coefficients will be equated wherefrom  the recursive formula \eqref{thm-main-stable-am} will be deduced.\\
\noindent
{\bf I Asymptotic expansion of the right-hand side of \eqref{def-stable-xt}.} 
From Theorem \ref{thm-main-RMNK}, with $a_m=a_m^M=a_m^N=a_m^K$, we have
\be\label{thm-stable-proof-am}
	a_m^R =\sum_{n=0}^{m} a_n \sum_{k=0}^{m-n}
			P[k,2n,\bd] P[m-n-k,-2n+1,\bs], \quad m\in\bN_0,
\ee
where
\be
	d_m = -\frac12 \sum_{n=0}^{m} a_n
		\sum_{k=0}^{2m+1-2n} P[k,2n,\bg]P[2m+1-2n-k,-2n+1,\bh], \quad m\in\bN_0,\label{thm-stable-proof-dm}
\ee
and
\be
	s_m = \frac12\sum_{n=0}^{m} a_n 	\sum_{k=0}^{2m-2n} 
		 P[k,2n,\bg]P[2m-2n-k,-2n+1,\bh], \quad m\in\bN_0. \label{thm-stable-proof-sm}
\ee
\noindent
{\bf II Extracting higher indexed term $a_m$.} 
In this step of the proof we shall detect the higher indexed term of the sequence $\ba$ contained in $a_m^R$. Simple computations reveal that $a_0^R=1$ and $a_1^R=a_1$ as can also be seen from the list of coefficients \eqref{list-coeff-Thm2-aR}. For $m\ge2$ we may split the sum on the right hand side of \eqref{thm-stable-proof-am} into three parts, $n=0$, $n\in\{1,\ldots,m-1\}$ and $n=m$:
\begin{align*}
	a_m^R 
		&= a_0 \sum_{k=0}^{m}P[k,0,\bd] P[m-k,1,\bs]\\			
			&\qquad+\sum_{n=1}^{m-1} a_n \sum_{k=0}^{m-n}P[k,2n,\bd] P[m-n-k,-2n+1,\bs]\\
			&\qquad+ a_m P[0,2m,\bd] P[0,-2m+1,\bs].
\end{align*}
According to Remark \ref{remark-powerlemma}, $P[k,0,\bd] =\delta_k$, $P[m,1,\bs] =s_m$,
	$P[0,2m,\bd]=d_0^{2m}$ and $P[0,-2m+1,\bs]=s_0^{-2m+1}$ and hence
\bes
	a_m^R  = a_0 s_m +\sum_{n=1}^{m-1} a_n \sum_{k=0}^{m-n}P[k,2n,\bd] P[m-n-k,-2n+1,\bs]
			  + a_m d_0^{2m} s_0^{-2m+1}.
\ees
Using formula 
\eqref{thm-stable-proof-sm}
for $s_m$ and substituting $a_0$, $d_0$ and $s_0$ with $1$, $\frac12$ and $1$ respectively, we obtain
\begin{align*}
	a_m^R  &=  \frac12\sum_{n=0}^{m} a_n \sum_{k=0}^{2m-2n}  P[k,2n,\bg]P[2m-2n-k,-2n+1,\bh]\\
			 &\qquad+\sum_{n=1}^{m-1} a_n \sum_{k=0}^{m-n}P[k,2n,\bd] P[m-n-k,-2n+1,\bs]
			 + 2^{-2m} a_m.
\end{align*}
We will split the first sum into three parts, $n=0$, $n\in\{1,\ldots,m-1\}$ and $n=m$, where for $n=0$ we have
\bes
	\frac12 a_0 \sum_{k=0}^{2m} P[k,0,\bg]P[2m-k,1,\bh]
	= \frac12 a_0 \sum_{k=0}^{2m} h_{2m}\delta_k
	= \frac12 a_0 \sum_{k=0}^{2m} a_{m}\delta_k
	= \frac12 a_m,
\ees
and for $n=m$ we have
\bes
	\frac12 a_m P[0,2m,\bg]P[0,-2m+1,\bh]= \frac12 a_m g_0^{2m} h_0^{-2m+1}=a_m 2^{-2m}.
\ees
Now we continue to analyze $a_m^R$ with those information included and terms with $a_m$ grouped together:
\be\label{racun-rm-5}
\begin{aligned}
	a_m^R  &= (2^{-1}+2^{-2m+1}) a_m\\
			&\qquad+\frac12\sum_{n=1}^{m-1} a_n \sum_{k=0}^{2m-2n} P[k,2n,\bg]P[2m-2n-k,-2n+1,\bh] \\
			 &\qquad+\sum_{n=1}^{m-1} a_n \sum_{k=0}^{m-n}P[k,2n,\bd] P[m-n-k,-2n+1,\bs].
\end{aligned}
\ee

Let $i_{max}(\cdot)$ denote the highest index $i$ such that $a_i$ appears within term inside the parenthesis. That is, 
\bes
	i_{max}(g_k)=\lfloor \frac{k+1}2 \rfloor,
	\qquad i_{max}(h_k)=\lfloor \frac{k}2 \rfloor.
\ees

In \eqref{racun-rm-5}, $P[k,2n,\bg]$ according to Lemma \ref{lemma-power} depends only on finite sequence $(g_0,\ldots,g_{k})$ and hence
\bes
	i_{max}(P[k,2n,\bg])=\max_{j\in\{0,\ldots,k\}} ( i_{max}(g_j)) = \lfloor \frac{k+1}2 \rfloor.
\ees
Also $P[2m-2n-k,-2n+1,\bh]$ from the same formula depends only on finite sequence $(h_0,\ldots,h_{2m-2n-k})$ and hence
\begin{align*}
	i_{max}&(P[2m-2n-k,-2n+1,\bh]) 
		  =\max_{j\in\{0,\ldots,2m-2n-k\}} ( i_{max}(h_j)) 
		=\lfloor m-n -\frac{k}2 \rfloor.
\end{align*}

The highest index that  appears in sum in the second row of \eqref{racun-rm-5} 
is
\begin{align*}
	&i_{max}\biggl(\sum_{n=1}^{m-1} a_n \sum_{k=0}^{2m-2n} P[k,2n,\bg]P[2m-2n-k,-2n+1,\bh] \biggr) \\
		&\quad \le \max_{\substack{n\in\{1,\ldots,m-1\} \\ 
					k\in\{0,\ldots,2m-2n\}}} 
					(m-1, i_{max}(P[k,2n,\bg]),i_{max}(P[2m-2n-k,-2n+1,\bh])\\
		&\quad = \max_{\substack{n\in\{1,\ldots,m-1\} \\ 
		k\in\{0,\ldots,2m-2n\}}} 
		\biggl(m-1,\lfloor \frac{k+1}2 \rfloor,\lfloor m-n -\frac{k}2 \rfloor\biggr)  
		=m-1.
\end{align*}

Regarding the third row of \eqref{racun-rm-5} first we observe $d_m$.
From formula \eqref{thm-stable-proof-dm} we have
\begin{align*}
	-2 d_m &= a_0 \sum_{k=0}^{2m+1} P[k,0,\bg]P[2m+1-k,1,\bh]\\
		 & \qquad+ \sum_{n=1}^{m} a_n \sum_{k=0}^{2m+1-2n} P[k,2n,\bg]P[2m+1-2n-k,-2n+1,\bh]\\
		 &= h_{2m+1} + \sum_{n=1}^{m} a_n \sum_{k=0}^{2m+1-2n} P[k,2n,\bg]P[2m+1-2n-k,-2n+1,\bh],
\end{align*}
and hence 
\bes
	i_{max}(d_m)\le\max(i_{max}(h_{2m+1}),m,i_{max}(g_{2m-1}),i_{max}(h_{2m-1}))=m.
\ees
From discussion before we may also see that
\bes
	i_{max}(s_m)=m.
\ees
Combining derived relations finally gives
\begin{align*}
	i_{max}&\biggl(\sum_{n=1}^{m-1} a_n 
		\sum_{k=0}^{m-n}P[k,2n,\bd] P[m-n-k,-2n+1,\bs] \biggr)\\
		&\qquad \le\max_{\substack{ 
			n\in\{1,\ldots,m-1\}\\k\in\{0,\ldots,m-n\}}}
			(m-1,i_{max}(d_k),i_{max}(s_{m-n-k}))\\
		&\qquad \le \max_{\substack{ 
			n\in\{1,\ldots,m-1\}\\k\in\{0,\ldots,m-n\}}}
			(m-1,k,m-n-k)\\
		&\qquad = m-1.
\end{align*}

\noindent
{\bf III Equating coefficients from the left-hand and right-hand side of \eqref{def-stable-xt}.}
For a stable mean $M$, the expansions of the left and right side in \eqref{def-stable-xt} must be equal, that is, $a_m=a_m^R$ for $m\in\bN_0$. Coefficient $a_0^R=1$ which is in agreement with property \eqref{def-mean-minmax}. Next, we have free coefficient $a_1$.
Furthermore, the connection between $a_m^R$ and $a_m$ for $m\ge2$ in \eqref{racun-rm-5} is also is linear so equating those coefficients \eqref{racun-rm-5} finally gives the relation \eqref{thm-main-stable-am}
which completes the proof.
\end{proof}


For a convenience, we give here first few coefficients $a_m^R$:
\be\label{list-coeff-aR-stable}
\begin{aligned}
	a_0^R&=1, \\
	a_1^R&=a_1,\\
	a_2^R&= \tfrac1{16}a_1(1+a_1)(1-4a_1)+\tfrac58a_2, \\
	a_3^R&= \tfrac1{64} ((1+a_1) (a_1 (1+2 a_1 (-3+6 a_1+8 a_1^2)-8 a_2)+6 a_2))+\tfrac{17}{32}a_3,\\
	a_4^R&= \tfrac1{256} (-56 a_1^5-48 a_1^6+33 a_2^2+24 a_1^4 (1+10 a_2)\\
		&\qquad +a_1^3 (22+300 a_2)+15 (a_2+a_3) +3 a_1^2 (-3+8 a_2+4 a_3) \\
		&\qquad+a_1 (1+3 a_2 (-7+32 a_2)+18 a_3))+ \tfrac{65}{128}a_4.
\end{aligned}
\ee

For a stable mean coefficients $a_m^R$ must be equal to $a_m$. Using \eqref{thm-main-stable-am} we obtain asymptotic expansion of a stable mean. With successive substitutions done all the subsequent coefficients can be seen as polynomials in variable $a_1$. Asymptotic expansion up to five terms of a bi-variate, symmetric, homogeneous stable mean in variables $(x-t,x+t)$ has the form:
\be\label{exp-Mtt-stable}
\begin{aligned}
	&M (x-t,x+t)= x+ a_1 t^2x^{-1} 
		+  \tfrac1{6}a_1(1+a_1)(1-4a_1) t^4x^{-3}  \\
		&\quad+ \tfrac1{90} a_1(1+a_1)(6-31a_1+36a_1^2+64a_1^3) t^6x^{-5}  \\
		&\quad+ \tfrac1{2520} a_1 (1+a_1) \bigl(90 - 531 a_1  + 937 a_1 ^2 + 568 a_1 ^3 - 3088 a_1 ^4 - 2176 a_1^5 \bigr) t^8x^{-7}  \\
		&\quad +\O(x^{-9}).
\end{aligned}
\ee

\begin{prop}\label{prop-tt-st}
	Bi-variate homogeneous symmetric mean $M$ with asymptotic expansion \eqref{mean-asymexp} has the asymptotic expansion \eqref{mean-asymexp-st} where for $s\neq \pm t$
	\bes
		a_m(s,t) = 2^{-m} \sum_{n=0}^{\lfloor \frac{m}2 \rfloor} 
		a_n \binom{1-2n}{m-2n} \left( {t-s}\right)^{2n}
			\left({t+s} \right)^{m-2n}, 
		\quad m\in\bN_0.
	\ees
\end{prop}
\begin{proof}
Let $\alpha=\frac{t+s}2$ and $\beta=\frac{t-s}2$. Then
\begin{align*}
	M (x+s,x+t) &= M\left(x+\a-\b,x+\a+\b\right)\\
	&= \sum_{n=0}^\infty a_n  \b^{2n} \left(x+\a\right)^{-2n+1}\\
	&= \sum_{n=0}^\infty a_n \b^{2n}x^{-2n+1}\sum_{k=0}^\infty\binom{-2n+1}{k}\a^kx^{-k}\\
	&= \sum_{m=0}^\infty \sum_{n=0}^{\lfloor \frac{m}2 \rfloor} 
		a_n \binom{1-2n}{m-2n} \b^{2n} \a^{m-2n} x^{-m+1}
\end{align*}
and the proof is complete.
\end{proof}

Combining Proposition \ref{prop-tt-st} with asymptotic expansion \eqref{exp-Mtt-stable} we get the following result.

\begin{cor}
	For a bi-variate homogeneous symmetric stable mean $M$ with asymptotic expansion \eqref{mean-asymexp} holds
	\bes
\begin{aligned}
	M&(x+s,x+t)= x+ \tfrac12(s+t)
		+\tfrac14 (s-t)^2 a_1  x^{-1} 
		-\tfrac18(s-t)^2(s+t) a_1 x^{-2} \\
		& + \tfrac1{16} \left( a_1  \left(s^2-t^2\right)^2-\tfrac{1}{6} a_1  (a_1 +1) (4 a_1 -1) (s-t)^4\right) x^{-3} \\
		& +\tfrac1{64} a_1  (s-t)^2 (s+t) \left((a_1 +1) (4 a_1 -1) (s-t)^2-2 (s+t)^2\right) x^{-4} \\
		& +\tfrac1{64} \Bigl( (s-t)^2 (s+t)^4 a_1 -(s-t)^4 (s+t)^2 a_1  (1+a_1 ) (-1+4 a_1 )\\
			&\quad + \tfrac1{90} (s-t)^6 a_1  (1+a_1 ) (6+a_1  (-31+4 a_1  (9+16 a_1 ))) \Bigr) x^{-5}
		 +\O(x^{-6}).
\end{aligned}
\ees
\end{cor}
\end{subsection}

\begin{subsection}{Stabilizable means}

\begin{thm}\label{thm-stabilizable}
Let homogeneous symmetric bi-variate means $M$, $N$ and $K$ have the asymptotic expansions \eqref{def-MNK-M}, \eqref{def-MNK-N} and \eqref{def-MNK-K}. Suppose $K$ and $M$ are stable means. Then the coefficients $(a_m^N)_{m\in\bN_0}$ in the asymptotic expansion \eqref{def-MNK-N} of $(K,M)$-stabilizable mean $N$ are given by:
\be\label{thm-stabilizable-coeff-amN}
\begin{aligned}
	a^N_0&=1,\\
	a_m^N &= \frac{2^{2m}}{2^{2m}-1}\biggl[ \frac12\sum_{n=0}^{m-1} a_n^N\sum_{k=0}^{2m-2n} P[k,2n,\bg^M]P[2m-2n-k,-2n+1,\bh^M]\\
		 	&\qquad+\sum_{n=1}^{m} a_n^K \sum_{k=0}^{m-n}P[k,2n,\bd] P[m-n-k,-2n+1,\bs] \biggr],\qquad m\in\bN,
\end{aligned}
\ee
where $\bg^M$ and $\bh^M$ are defined in \eqref{def-bgbh1} with $a_m=a_m^M$ and $\bd$ and $\bs$ are defined by \eqref{def-dm} and \eqref{def-sm}.
\end{thm}

\begin{proof}
From Theorem \ref{thm-main-RMNK}, with $R=N$ and thereby $a_m^R=a_m^N$, $m\in\bN_0$, we have
\be\label{thm-stabilizable-proof-aN-1}
	a_m^N =\sum_{n=0}^{m} a_n^K \sum_{k=0}^{m-n}
			P[k,2n,\bd] P[m-n-k,-2n+1,\bs], \quad m\in\bN_0.
\ee
For $m=0$ in equation \eqref{thm-stabilizable-proof-aN-1} we obtain that $a^N_0=1$ as expected.
Now let $m>0$. First, we split the sum in two parts, for $n=0$ and $n\in\{1,\ldots,m\}$:
\bes
		a_m^N 	= a_0^K \sum_{k=0}^{m}P[k,0,\bd] P[m-k,1,\bs]
				+\sum_{n=1}^{m} a_n^K \sum_{k=0}^{m-n}P[k,2n,\bd] P[m-n-k,-2n+1,\bs].
\ees
According to Remark \ref{remark-powerlemma}, the first part reduces and we have
\bes
		a_m^N 	= s_m+\sum_{n=1}^{m} a_n^K \sum_{k=0}^{m-n}P[k,2n,\bd] P[m-n-k,-2n+1,\bs].
\ees
Now we use formula \eqref{def-sm} for $s_m$
\begin{align*}
		a_m^N 	&= \frac12\sum_{n=0}^{m} a_n^N 	\sum_{k=0}^{2m-2n} P[k,2n,\bg^M]P[2m-2n-k,-2n+1,\bh^M]\\
		 	&\qquad+\sum_{n=1}^{m} a_n^K \sum_{k=0}^{m-n}P[k,2n,\bd] P[m-n-k,-2n+1,\bs],
\end{align*}
then split the sum into two parts, for $n=m$ and $n\in\{0,\ldots,m-1\}$, and apply conclusions from Remark \ref{remark-powerlemma}. That is,
\begin{align*}
		a_m^N 
		 	&=\frac12 a_m^N P[0,2m,\bg^M]P[0,-2m+1,\bh^M]\\
		 	&\qquad+  \frac12\sum_{n=0}^{m-1} a_n^N\sum_{k=0}^{2m-2n} P[k,2n,\bg^M]P[2m-2n-k,-2n+1,\bh^M]\\
		 	&\qquad+\sum_{n=1}^{m} a_n^K \sum_{k=0}^{m-n}P[k,2n,\bd] P[m-n-k,-2n+1,\bs],
\end{align*}
which with $g_0=1$ and $h_0=2$ reads as
\begin{align*}
		a_m^N
		 	&=2^{-2m} a_m^N 
		 	+  \frac12\sum_{n=0}^{m-1} a_n^N\sum_{k=0}^{2m-2n} P[k,2n,\bg^M]P[2m-2n-k,-2n+1,\bh^M]\\
		 	&\qquad+\sum_{n=1}^{m} a_n^K \sum_{k=0}^{m-n}P[k,2n,\bd] P[m-n-k,-2n+1,\bs]
\end{align*}
wherefrom \eqref{thm-stabilizable-coeff-amN} easily follows.
\end{proof}

Using formula \eqref{thm-stabilizable-coeff-amN} from Theorem \ref{thm-stabilizable} we obtain the coefficients of a mean $N$ such that \eqref{def-stabilizable} holds for given means $K$ and $M$. Here is a list of first few such coefficients:
\begin{align*}
 a^N_0 &=1,\\
 a^N_1 &=\frac13 \big( a^K_1+2a^M_1 \big), \\ 
 a^N_2 &=\frac{1}{45} \big(-2 a^K_1 (6 a^M_1+5) a^M_1-(a^K_1)^2 (8 a^M_1+3)+a^K_1+3 a^K_2\\
 	&\qquad+8 (a^M_1)^3+4 (a^M_1)^2+2 a^M_1+24 a^M_2\big) \\
 a^N_3 &= \frac1{2835} \big( 
 	(a^K_1)^3 \bigl(8 a^M_1 (26 a^M_1+23)+41\bigr) \\
 	&\qquad +2 (a^K_1)^2 \bigl(a^M_1 (4 a^M_1 (40 a^M_1+81)+61) -5 (48 a^M_2+7)\bigr) \\
 	&\qquad -a^K_1 \bigl(18 a^K_2 (16 a^M_1+7)+408 a^M_2+16 a^M_1 (2 a^M_1 (a^M_1 (3 a^M_1-7)-1)\\
 	&\qquad+54 a^M_2+11)-21\bigr) 
 	  -6 a^K_2 \bigl(a^M_1 (68 a^M_1+71)-3\bigr)+45 a^K_3\\
 	  &\qquad +6 a^M_1 \bigl(32 (a^M_1)^4-12 (a^M_1)^2  
 	 +16 (16 a^M_1+7) a^M_2+7\bigr) 
 	 +144 (a^M_2+10 a^M_3) \big).
\end{align*}

If we require that mean $N$ is $(K,M)$-stabilizable then we must take into account stability of $K$ and $M$.
When means $K$ and $M$ are stable, coefficients $a_m^K$ and $a_m^M$, $m\ge2$, can be expressed as polynomials in variables $a_1^K$ and $a_1^M$ respectively. From expansion \eqref{exp-Mtt-stable} used with coefficients $a_m^K$ and  $a_m^M$ we obtain coefficients $a_m^N$ with $K$ and $M$ stable:
\be\label{thm-stabilizable-list-coeff-2}
\begin{aligned}
 a^N_0 &=1,\\
 a^N_1 &=\frac13 \left(a^K_1+2a^M_1 \right), \\ 
 a^N_2 &= \frac{1}{90} \big( -(a^K_1)^2 (16 a_1^M+9)+a_1^K(3-4a_1^M(6a_1^M+5))\\
 	&\qquad-4(a_1^K)^3-4a_1^M (4a_1^M (a_1^M+1)-3)\big)\\
 a^N_3 &=\frac{1}{5670} \big(64 (a^K_1)^5 + 4 (a_1^K)^4 (67 + 96 a_1^M)\\
 	&\qquad + 12 a_1^M(-1 + 2 a_1^M) (3 + 2 a_1^M) (-9 + 8 a_1^M (1 + a_1^M)) \\
  &\qquad+3 (a_1^K)^3 (63 + 8 a_1^M (51 + 40 a_1^M)) \\
  &\qquad+ (a_1^K)^2 (-225 + 2 a_1^M (207 + 4 a_1^M (273 + 160 a_1^M))) \\
  &\qquad+2 a_1^K (27 + a_1^M (-315 + 8 a_1^M (3 + 4 a_1^M (29 + 15 a_1^M)))) \big).
\end{aligned}
\ee
\end{subsection}

\begin{subsection}{Stabilized means}

\begin{thm}\label{thm-stabilized}
Let homogeneous symmetric bi-variate means $M$, $N$ and $K$ have the asymptotic expansions \eqref{def-MNK-M}, \eqref{def-MNK-N} and \eqref{def-MNK-K}. Suppose $K$ and $N$ are stable means. Then the coefficients $(a_m^M)_{m\in\bN_0}$ in the asymptotic expansion \eqref{def-MNK-M} of $(K,N)$-stabilized mean $M$ are given by:
\be\label{thm-formula-stabilized-aM}
\begin{aligned}
	a_0^M&=1,\\
	a_m^M &=\sum_{n=1}^{m} a_n^N\sum_{k=0}^{2m-2n} P[k,2n,\bg^M]P[2m-2n-k,-2n+1,\bh^M]\\
		 &\qquad+ 2\sum_{n=1}^{m} a_n^K \sum_{k=0}^{m-n}P[k,2n,\bd] P[m-n-k,-2n+1,\bs],\quad m\in\bN,
\end{aligned}
\ee
where $\bg^M$ and $\bh^M$ are given by \eqref{def-bgbh1} with $a_m=a_m^M$ and $\bd$ and $\bs$ are defined by \eqref{def-dm} and \eqref{def-sm}.
\end{thm}
\begin{proof}
Using formula \eqref{thm-R-formula-aR} from Theorem \ref{thm-main-RMNK}, with $R=M$ and thereby $a_m^R=a_m^M$, $m\in\bN_0$, we obtain
\bes
	a_m^M =\sum_{n=0}^{m} a_n^K \sum_{k=0}^{m-n} P[k,2n,\bd] P[m-n-k,-2n+1,\bs].
\ees
For $m=0$ we obtain $a_0^M=1$ as expected. Now let $m>0$. On the right hand side for $n=0$ we have 
	$a_0^K \sum_{k=0}^{m} P[k,0,\bd] P[m-k,1,\bs]$ where $a_0^K=1$, $P[k,0,\bd]=\delta_k$ and $P[m,1,\bs]=s_m$. Hence, the sum can be written as
\bes
	a_m^M 	=s_m + \sum_{n=1}^{m} a_n^K \sum_{k=0}^{m-n}P[k,2n,\bd] P[m-n-k,-2n+1,\bs].
\ees
After replacing $s_m$ according to definition \eqref{def-sm} and splitting it into two parts, for $n=0$ and $n\ge1$, we obtain
\begin{align*}
		a_m^M 	&=\frac12\sum_{n=0}^{m} a_n^N\sum_{k=0}^{2m-2n} P[k,2n,\bg^M]P[2m-2n-k,-2n+1,\bh^M]\\
		 &\qquad+ \sum_{n=1}^{m} a_n^K \sum_{k=0}^{m-n}P[k,2n,\bd] P[m-n-k,-2n+1,\bs]\\
		&=\frac12 a_m^M 
			+ \frac12\sum_{n=1}^{m} a_n^N\sum_{k=0}^{2m-2n} P[k,2n,\bg^M]P[2m-2n-k,-2n+1,\bh^M]\\
		&\qquad + \sum_{n=1}^{m} a_n^K \sum_{k=0}^{m-n}P[k,2n,\bd] P[m-n-k,-2n+1,\bs]
\end{align*}
which after subtracting $\frac12 a_m^M $ and multiplying by 2 gives the desired result \eqref{thm-formula-stabilized-aM}.
\end{proof}

Using formula \eqref{thm-formula-stabilized-aM} from Theorem \ref{thm-stabilized} we obtain coefficients in the asymptotic expansion of a mean $M$ such that equation \eqref{def-stabilized} holds for a given means $K$ and $N$. Here are first few of them:
\begin{align*}
 a^M_0 &=1,\\
 a^M_1 &=\frac12 \big(a_1^K+a_1^N), \\ 
 a^M_2 &=\frac18 (a^K_2+a^N_1+(a^N_1)^2+(a^N_1)^3-a^K_1 a^N_1 (3+2 a^N_1)-(a^K_1)^2 (1+3 a^N_1)+a^N_2\big),\\
 a^M_3 &=\frac1{32} (a^K_3+(a^K_1)^3 (1+2 a^N_1) (2+5 a^N_1)+(a^K_1)^2 (a^N_1 (5+6 a^N_1 (3+a^N_1))-2 a^N_2) \\
 	&\quad-a^K_1 (2 a^K_2 (2+5 a^N_1)+a^N_1 (5+a^N_1 (2+a^N_1+2 (a^N_1)^2)-2 a^N_2)-a^N_2) +6 a^N_2\\
 	&\quad+a^N_1 (1-3 a^K_2 (3+2 a^N_1)+10 a^N_2+a^N_1 (a^N_1+2 (a^N_1)^2 (1+a^N_1)+8 a^N_2))+a^N_3).
\end{align*}

In order $M$ to be $(K,N)$-stabilized, means $K$ and $N$ need to be stable and therefore their coefficients obey the rule of stable means coefficients \eqref{exp-Mtt-stable} and can be expressed trough $a_1^K$ and $a_1^N$. Here are the first few coefficients in the asymptotic expansion of $(K,N)$-stabilized mean $M$:

\be\label{thm-stabilized-list-coeff-2}
\begin{aligned}
 a^M_0 &=1,\\
 a^M_1 &=\frac12 \big(a_1^K+a_1^N), \\ 
 a^M_2 &=\frac1{48}\big(-4 (a^K_1)^3 - 9 (a^K_1)^2 (1 + 2 a^N_1)  \\
  	&\qquad + a^K_1 (1 - 6 a^N_1 (3 + 2 a^N_1)) +a^N_1 (7 + a^N_1 (3 + 2 a^N_1))\big), \\
 a^M_3 &= \frac{1}{2880} \big(64 (a^K_1)^5 + 20 (a^K_1)^4 (17 + 30 a^N_1) \\
 	&\qquad + 5 (a^K_1)^3 (73 + 36 a^N_1 (10 + 7 a^N_1))\\
  &\qquad+ 5 (a^K_1)^2 (-17 + 3 a^N_1 (45 + 44 a^N_1 (3 + a^N_1)))\\
  &\qquad - 3 a^K_1 (-2 + 5 a^N_1 (38 + a^N_1 (19 + 4 a^N_1 (4 + 5 a^N_1))))\\
  &\qquad - a^N_1 (-186 + a^N_1 (145 + a^N_1 (595 + 4 a^N_1 (170 + 59 a^N_1))))\big).
\end{aligned}
\ee
\end{subsection}
\end{section}

\begin{section}{Simultaneously stabilizable and stabilized} \label{section-simultaneous}
In this section we will show how Theorems from Section 3 may be used to derive some necessary conditions in following interesting cases.

\setcounter{subsection}{0}
\begin{subsection}{Simultaneously $(K,M)$-stabilizable and $(M,K)$-stabilizable}

Let $K$ and $M$ be two stable means. If $N$ is $(K,M)$-stabilizable and $(M,K)$-stabilizable, then
\bes
	\frac13(a_1^K+2a_1^M)=a_1^N =\frac13(a_1^M+2a_1^K)
\ees
wherefrom it follows that $a_1^N=a_1^M=a_1^K$ and, because asymptotic expansion of stable mean is completely determined by the coefficient with index 1, it holds that $a_m^M=a_m^K$ for all $m\in\bN_0$.
Calculating the next coefficients from the list \eqref{thm-stabilizable-list-coeff-2} reveals that
\begin{align*}
	a_2^N&=\frac16 a_1^N (1+a_1^N) (1-4 a_1^N),\\
	a_3^N&=\frac1{90} a_1^N (1+ a_1^N) (6-31 a_1^N +36 (a_1^N)^2 +64(a_1^N)^3),
\end{align*}
which have the form of stable mean coefficients \eqref{exp-Mtt-stable} suggesting $N$ should also be stable.

\end{subsection}

\begin{subsection}{Simultaneously $(K,N)$-stabilized and $(N,K)$-stabilized}
Let $K$ and $N$ be two stable means. Assume mean $M$ is simultaneously $(K,N)$-stabilized and $(N,K)$-stabilized. Observe the list of coefficients \eqref{thm-stabilized-list-coeff-2}. Coefficient $a_1^M$ is symmetric in $K$ and $N$.
For $m=2$ the following equality must hold
	\begin{align*}
		a_2^M&=\frac1{48} (-4 (a^K_1)^3-9 (a^K_1)^2 (1+2 a^N_1)+a^K_1 (1-6 a^N_1(3+2 a^N_1))\\
			&\qquad+a^N_1 (7+a^N_1(3+2 a^N_1)))\\
			&=\frac1{48} (-4 (a^N_1)^3-9 (a^N_1)^2 (1+2 a^K_1)+a^N_1 (1-6 a^K_1 (3+2 a^K_1))\\
				&\qquad+a^K_1 (7+a^K_1(3+2 a^K_1)))
	\end{align*}
wherefrom it follows
\bes
	-\frac18 (a^K_1 - a^N_1) (1 + a^K_1 + a^N_1)^2=0.
\ees
We have two possibilities: $a^K_1=a^N_1$ and $a^K_1+a^N_1=-1$. 
If the first one holds, then
	\begin{align*}
		a^M_1 &= a^N_1=a^K_1,\\
		a^M_2 &
		=\frac16 a^M_1 (1+a^M_1) (1-4 a^M_1),
	\end{align*}
which suggests that $M$ is should be stable and therefore (asymptotically) equal to $K$ and $N$. 
If the second condition holds, then simple computation yield
$$a^M_2=-\frac1{2},\ a^M_2=-\frac1{8},\ a^M_3=-\frac1{16},\ a^M_4=-\frac5{128},\
a_5^M=-\frac7{256}.$$
These coefficients correspond to first coefficients in the asymptotic expansion of geometric mean which can be found in \cite{ElVu-2014-09}. This correspondence suggests that geometric mean is the only simultaneously $(K,N)$-stabilized and $(N,K)$-stabilized for  different stable means $K$ and $N$. In that case an interesting equation appears. If $G$ is $(K,N)$-stabilized, using homogeneity of $N$ and $K$
\begin{align*}
	G(s,t)&=K\bigl( N(s,\sqrt{st}),N(\sqrt{st},t) \bigr) 
		 =K\bigl(\sqrt{s} N(\sqrt{s},\sqrt{t}),\sqrt{t} N(\sqrt{s},\sqrt{t}) \bigr)\\
		&= N(\sqrt{s},\sqrt{t})  K(\sqrt{s},\sqrt{t})
\end{align*}
which implies that $G$ is also $(N,K)$-stabilized.
Additionally, by including $s^2$ and $t^2$ instead of $s$ and $t$, for $s,t>0$, we obtain
\bes
	st=G(s^2,t^2)=N(s,t)K(s,t)
\ees
and taking the square root the following equation follows
\bes
	G=G(N,K)
\ees
which is also known as Gauss functional equation which defines compound mean obtained with Gauss iterative procedure (\cite[Ch.VI.3]{Bull-2003}) indicating that geometric mean $G$ is also compound mean of $N$ and $K$, i.e.\ $G=N\otimes K$.

\end{subsection}

\begin{subsection}{Simultaneously $(K,N)$-stabilizable and $(K,N)$-stabilized}

Modifying formula \eqref{def-stabilizable} by interchanging $M$ and $N$, from the list \eqref{thm-stabilizable-list-coeff-2} we read first few coefficients of $(K,N)$-stabilizable mean $M$ and by comparison with coefficients from list \eqref{thm-stabilized-list-coeff-2} we obtain the necessary conditions for mean $M$ to be $(K,N)$-stabilizable and also $(K,N)$-stabilized. For $m=1$ we have 
\bes
	\frac13 \left(a^K_1+2a^N_1 \right)= a^M_1=\frac12 \big(a_1^K+a_1^N)
\ees
 and hence $a_1^K=a_1^N= a^M_1$. Since $K$ and $N$ are stable means and their asymptotic expansions are completely determined by $a_1^K$ and $a_1^N$ it follows that $a_m^K=a_m^N$ for all $m\in\bN_0$.
Then, either by using that $M$ is $(K,N)$-stabilizable or $M$ is $(K,N)$-stabilized, after replacing $a_1^K$ and $a_1^N$ by $a^M_1$, we obtain coefficients:
\begin{align*}
 a^M_2 &= \frac16 a^M_1 (1 + a^M_1) (1 - 4 a^M_1),\\
 a^M_3 &=\frac1{90} a_1^M (1+ a_1^M) (6-31 a_1^M +36 (a_1^M)^2 +64(a_1^M)^3),
\end{align*}
 which are stable mean coefficients indicating that $M$ should also be stable and (asymptotically) equal to $K$ and $N$.
\end{subsection}
\end{section}

\begin{section}{Examples and applications} \label{section-examples}
Based on the results of Theorem \ref{thm-main}, we may easily see which are the necessary conditions for some mean to be stable. This can be useful especially in the case of parametric means such as Stolarsky and Gini means mentioned in the Introduction. In the paper \cite{ElVu-2014-04} we have derived asymptotic expansion of some one and two parameter classes of means which will be used to solve the open problem of Ra\"{i}ssouli and to demonstrate the application of the main Theorem. More about the mathematical means  can be found in  \cite{Bull-2003}.

\setcounter{subsection}{0}
\begin{subsection}{Power means}

The $r$-th power mean is defined for all $s,t>0$ by
        \bes
                B_r(s,t)=
		\begin{cases}
		\displaystyle
		\biggl(\frac{s^r+t^r}{2}\biggr)^{1/r},\qquad &r\ne0,\\
		\sqrt{st},&r=0.
		\end{cases}
        \ees
	For example, the special cases of this mean are arithmetic mean $A=B_1$, quadratic mean $Q=B_2$ and harmonic mean $H=B_{-1}$. Geometric mean $G=B_0$ is obtained as limit case of $B_r$ as $r\to0$.
	Easy computations reveal that power mean $M_r$ is stable for every $r$.
	
In paper \cite{ElVu-2014-04}, with $\alpha=0$ and $\beta=t$, we find the asymptotic expansion of the $r$-th power mean:
\be\label{examples-exp-power}
	B_r(x-t,x+t)\sim x +\tfrac12(r-1)t^2 x^{-1}-\tfrac{1}{24}(r-1)(r+1)(2r-3)t^4x^{-3}+\O(x^{-5}).
\ee	
Asymptotic behavior of $n$-variable power means has been studied in \cite	{ElVu-2016-01}.
Since power mean is stable these coefficients satisfy the recursive formula \eqref{thm-main-stable-am}.
\end{subsection}

\begin{subsection}{Gini means}
The Gini means are defined for all $s,t>0$ by
\bes
	G_{p,r}(s,t)= \begin{cases}
		\left( \dfrac{s^p+t^p}{s^r+t^r} \right)^{\frac{1}{p-r}},& p\neq r,\\
		\exp \left( \dfrac{s^p\log s +t^p\log t}{s^p+t^p} \right), &p=r\neq 0,\\
		\sqrt{st},&p=r=0.
	\end{cases}
\ees
for parameters $p$ and $r$. These means were first introduced by  Gini (\cite{Gini-1938}). 
Power means belong to the class of power means as $G_{0,r}=B_r$ and also Lehmer mean $G_{r+1,r}$ is  special case of Gini mean.

In paper \cite{ElVu-2014-04}, with $\alpha=0$ and $\beta=t$, we find the asymptotic expansion of Gini means:
\bes
	\begin{aligned}
	G_{p,r}(x-t,x+t)&\sim x+\tfrac12(p+r-1)t^2 x^{-1} 
		+\tfrac{1}{24} \bigl[(-3-2 p^3+p^2 (3-2 r)\\
		&\qquad+2 r+(3-2 r) r^2+p (2-2 (-3+r) r))\bigr] t^4x^{-3}+\O(x^{-5}).
		\end{aligned}
\ees
By equating known coefficients of the asymptotic expansion of Gini means with coefficients of stable mean \eqref{exp-Mtt-stable}, we see that 
\begin{align*}
	a_1&= \tfrac12(p+r-1)\\
	\tfrac1{6}a_1(1+a_1)(1-4a_1) &=\tfrac{1}{24} \bigl[(-3-2 p^3+p^2 (3-2 r) 
		 +2 r \\
		  &\qquad +(3-2 r) r^2+p (2-2 (-3+r) r))\bigr]
\end{align*}
which is equivalent to
\bes
	pr(p+r)=0.
\ees
Since $G_{0,r}=B_r$, $G_{p,0}=B_p$ and $G_{p,-p}=B_0=G$, we may conclude that the only stable Gini means are power means.
\end{subsection}

\begin{subsection}{Stolarsky means}
	The Stolarsky means, also called the extended means or difference mean values, is a class of two-parameter means introduced by Stolarsky in \cite{Stol-1975}. Their properties were studied by Leach and Sholander in \cite{LeachSh-1978,LeachSh-1983} and further by P\'ales \cite{Pal-differences} and others.
	
	The Stolarsky mean of order $p,r$ is defined for all $s,t>0$ by
	\bes
		E_{p,r}(s,t)=\begin{cases}
			\left[\frac{r(t^p-s^p)}{p(t^r-s^r)}\right]^{1/(p-r)}, & p\ne r,\ p,r\ne0,\\
			\frac1{e^{1/r}}\left(\frac{t^{t^r}}{s^{s^r}}\right)^{1/(t^r-s^r)},& r=p\ne0,\\
			\left[\frac{t^r-s^r}{r(\log t-\log s)}\right]^{1/r},& p=0,\ r\ne0,\\
			\sqrt{st}, &p=r=0.
			\end{cases}
	\ees
	
	Special cases are obtained by limit procedure.
	It is symmetric both on $t$ and $s$ as well on $p$ and $r$.
	$E_{p,r}(s,t)$ increases with increase in either $s$ or $t$ and also with increase in either $r$ or $s$ (\cite{LeachSh-1978}).

	From paper \cite{ElVu-2014-04}, with $\alpha=0$ and $\beta=t$, we have:
\bes
	\begin{aligned}
	E_{p,r}(x-t,x+t)&\sim x+\tfrac16(p+r-3)t^2 x^{-1} 
		+\tfrac{1}{360} \bigl[-45-2 p^3+p^2 (5-2 r)\\
		&\qquad +r (10+(5-2 r) r)-2 p (-5+(-5+r) r) \bigr] t^4x^{-3}+\O(x^{-5}).
	\end{aligned}
\ees

 By comparing corresponding coefficients of Stolarsky with coefficients of stable mean, we see that
\begin{align*}
	a_1&= \tfrac16(p+r-3)\\
	\tfrac1{6}a_1(1+a_1)(1-4a_1) &=\tfrac{1}{360} \bigl[-45-2 p^3+p^2 (5-2 r) \\
		&\qquad+r (10+(5-2 r) r)-2 p (-5+(-5+r) r) \bigr]
\end{align*}
which reduces to
\bes
	(p-2r)(2r-q)(p+r)=0.
\ees
Since $E_{2r,r}=B_r$, $E_{p,2p}=B_p$ and $E_{p,-p}=B_0=G$, we may conclude that the only stable Stolarsky means are again power means.
\end{subsection}

\begin{subsection}{Generalized logarithmic mean}

        Let $r$ be a real number. The generalized logarithmic mean (\cite{Stol-1980}) is defined for $s,t>0$ ($s\neq t$) by 
    \bes
		L_r(s,t)=\begin{cases}
			\biggl(\frac{t^{r+1}-s^{r+1}}{(r+1)(t-s)}\biggr)^{1/r}, &r\ne -1,0,\\
			\frac{t-s}{\log t-\log s},&r=-1,\\
			\frac1e\biggl(\frac{t^t}{s^s}\biggr)^{1/(t-s)},&r=0.
			\end{cases}
	\ees
with $L_{-1}$ being the logarithmic and $L_0$ the identric mean.
It is Schur-convex for $r>1$ and Schur-concave for $r<1$ (\cite{ElPec-2000a}).

    The beginning of the asymptotic expansion of generalized logarithmic mean (\cite{ElVu-2014-04}) reads as 
    \bes
    	L_r(x-t,x+t)\sim x+\tfrac16(r-1)t^2x^{-1}-\tfrac1{360} (r-1)(2r^2+5r-13)t^4x^{-3}+\O(x^{-5}).
    \ees
In order $L_r$ to be stable its coefficients must coincide with stable mean coefficients \eqref{exp-Mtt-stable}, that is, the following conditions must hold
\begin{align*}
	a_1&=\tfrac16(r-1),\\
	 \tfrac1{6}a_1(1+a_1)(1-4a_1) &= -\tfrac1{360} (r-1)(2r^2+5r-13),
\end{align*}
which leads to the cubic equation with solutions $r=1$, $r=-\frac12$ and $r=-2$. In each of these three cases we have one of  the power means, that is, $L_1=A$, $L_{-\frac12}=B_{\frac12}$ and $L_{-2}=G$ which are all stable.

 \begin{rem}
 	Within classes of Gini, Stolarsky and generalized logarithmic mean only power means are stable.
 \end{rem}
\end{subsection}

\begin{subsection}{Stability and stabilizability with power means}\label{examples-subsection-power}
 Let us find coefficients for $(K,M)$-stabilizable mean $N$, where $K=B_p$ and $M=B_q$.
 Combining coefficients of stabilizable mean \eqref{thm-stabilizable-list-coeff-2} with those from \eqref{examples-exp-power} we find that coefficients in the expansion of mean $N$ which is stabilizable with pair of power means $(B_p,B_q)$ are:
 	\be\label{examples-list-MpMq-stabilizable}
 	\begin{aligned}
 		a_0^N &= 1,\\
 		a_1^N &= \frac16 (p+2 q-3),\\
 		a_2^N &= \frac1{360} (-45-2 p^3+p^2 (5-8 q)+2 p (5+2 (5-3 q) q)\\
 			&\qquad+4 q (5+(5-2 q) q)).
 	\end{aligned}
 	\ee

 Observe the list of coefficients \eqref{thm-stabilized-list-coeff-2}, where $K=B_p$ $N=B_q$ and whose coefficients can be obtained from \eqref{examples-exp-power}. Then we obtain the coefficients of mean $M$ which is stabilized with pair of power means $(B_p,B_q)$:
  	\begin{align*}
 		a_0^M &= 1,\\
 		a_1^M &= \frac14 (p+q-2),\\
 		a_2^M &= \frac1{192} (-24-2 p^3+p^2 (6-9 q)+q (2+q) (4+q)+p (8-6 (-2+q) q)).
 	\end{align*}
\end{subsection}

\begin{subsection}{Seiffert and Neuman-S\'andor means}
 Let $s,t>0$. The first and the second Seiffert means (\cite{Bull-2003}) are defined by 
 	\bes
 		P(s,t)  =\frac{t-s}{2 \arcsin\frac{t-s}{t+s}},\qquad
 		T(s,t)  =\frac{t-s}{2 \arctan\frac{t-s}{t+s}},
	\ees
	and Neuman-S\'andor mean (\cite{NeuSan-2003-SB1}) is defined by 
	\bes
		\NS(s,t) =\frac{t-s}{2 \arcsinh\frac{t-s}{t+s}}.
	\ees
	
	Asymptotic expansions of Seiffert means can be found in \cite{Vu}, with $\alpha=0$ and $\beta=t$:
	\begin{align}
		P(x-t,x+t) &\sim x-\frac16 tx^{-1} -\frac{17}{360} t^4 x^{-3} 
			-\frac{367}{15120}t^6 x^{-5} +\O(x^{-7}),\label{examples-exp-Seiffert1}\\
		T(x-t,x+t) &\sim x+\frac13 tx^{-1} -\frac{4}{45} t^4 x^{-3} 
			+\frac{44}{945}t^6 x^{-5} +\O(x^{-7}),\notag
	\end{align}	
	and the asymptotic expansion of the Neuman-S\'andor mean was given in \cite{ElVu-2014-20}:
	\be
		\NS(x-t,x+t) \sim x+\frac16 tx^{-1} -\frac{17}{360} t^4 x^{-3} 
			+\frac{367}{15120}t^6 x^{-5} +\O(x^{-7}).\label{examples-exp-NS}
	\ee
	Assume $P$ is $(K,M)$-stabilizable mean. Then comparing coefficients \eqref{examples-exp-Seiffert1} with those from list \eqref{thm-stabilizable-list-coeff-2} yields
	\bes
		-\frac16=\frac13(a_1^K+2a_1^M),
	\ees
	and with $a_1^K=-\frac12-2a_1^M$, comparison of the second coefficients yields
	\bes
		-\frac{17}{360}=-\frac1{360}(13 +32(a_1^M)^2)
	\ees
	or equivalently
	$
		a_1^M = \pm \frac1{\sqrt8}.
	$
	Each of these two values leads to contradiction when comparing the third coefficients.
	
	By the same procedure we may obtain that
	the second Seiffert and Neuman-S\'andor mean are also not stabilizable.
\end{subsection}
	
\begin{subsection}{Asymptotic inequalities and sub-stabilizability with power means}

Let us observe the difference between the first Seiffert mean $P$ and the resultant mean-map $\cR(K,N,M)$ with $K=B_p$ and $M=B_q$. Since for $(K,M)$-stabilizable mean $N$, $a_m^R$ is equal to $a^N_m$, we may observe the difference between the asymptotic expansion of the first Seiffert and mean with coefficients from Subsection \ref{examples-subsection-power}.
According to \eqref{examples-exp-Seiffert1} and \eqref{examples-list-MpMq-stabilizable}, its asymptotic expansion reads as
\begin{align*}
	 P&(x-t,x+t)-\cR(B_p,P,B_q)(x-t,x+t) 
	 	 =P(x-t,x+t)-R(x-t,x+t) \\
		& \sim - \left(\frac16+\frac16(p+2q-3)\right) t^2 x^{-1} 
		  - \biggl(\frac{17}{360}+\frac1{360} (-45-2 p^3+p^2 (5-8 q) \\
			&\qquad \qquad+2 p (5+2 (5-3 q) q)+4 q (5+(5-2 q) q))\biggr)t^4 x^{-3}+\O(x^{-5}).
\end{align*}
The best approximation is obtained when the coefficient in the first parentheses is equal to zero, i.e.\ when $p=2-2q$. For such relation between $q$ and $p$ the asymptotic expansion of a difference reads as
\begin{align*}
	P(x-t,x+t)&-\cR(B_p,P,B_q)(x-t,x+t)\\
		& \sim \frac1{90} (1-4q +2q^2) t^4 x^{-3}\\
		&\quad+\frac1{11340}\left(99-2 q(q-2)(-83+12 (-2+q) q) \right)  t^6 x^{-5} 
		+\O(x^{-7}).
\end{align*}
Again, from equating the first coefficient with zero, we obtain two solutions of quadratic equation: $q_{1,2}=1\pm \frac{\sqrt2}2$. For either one of these solutions, the asymptotic expansion of a difference is
\begin{align*}
	P(x-t,x+t)-\cR(B_p,P,B_q)(x-t,x+t) 
		 \sim \frac1{1134}t^6 x^{-5} +\O(x^{-7}),
\end{align*}
and thereby the difference is asymptotically greater than 0:
\bes
	P-\cR(B_p,P,B_q) \succ0,
\ees
which means that the necessary condition for the inequality $P>\cR(B_p,P,B_q)$ is fulfilled.
If $P$ is $(B_p,B_q)$-sub-stabilizable for $(p,q)=(\pm\sqrt2,1\mp \frac{\sqrt2}2)$, then those parameters are optimal. Numerical experiments and plotting the graph of a difference $P(s,1-s)-\cR(B_p,P,B_q)(s,1-s)$, $s\in[0,1]$, indicate that mean $P$ should be $(B_p,B_q)$-sub-stabilizable for these $p$ and $q$.

Observe the difference between the Neuman-S\'andor mean $\NS$ and the resultant mean-map $\cR(K,\NS,M)$ with $K=B_p$ and $M=B_q$. Similarly as before, using \eqref{examples-exp-NS} and \eqref{examples-list-MpMq-stabilizable}, its asymptotic expansion reads as
\begin{align*}
	 \NS(x-t,x+t)&-\cR(B_p,\NS,B_q)(x-t,x+t)\\
		& \sim \frac16 ( 4-2p-2q ) t^2 x^{-1} 
		  - \frac1{360}\bigl(28-10 p+2 p^3+4 p q (-5+3 q)\\
		  &\qquad +p^2 (-5+8 q)+4 q (-5+q (-5+2 q))\bigr)
		  t^4 x^{-3}+\O(x^{-5}).
\end{align*}
The best approximation is obtained when the coefficient in the first parentheses is equal to zero, i.e.\ when $p=4-2q$. For such relation between $q$ and $p$ the asymptotic expansion of a difference reads as
\begin{align*}
	\NS(x-t,x+t)&-\cR(B_p,\NS,B_q)(x-t,x+t) 
		  \sim \frac1{90} (9-16q+4q^2) t^4 x^{-3}\\
		&+\frac1{11340}\bigl(-123-4 (-4+q) q (67+12 (-4+q) q)\bigr)  t^6 x^{-5} 
		+\O(x^{-7}).
\end{align*}
Again, from equating the next coefficient with zero, we obtain two solutions of quadratic equation: $q_{1,2}=2\pm \frac{\sqrt7}2$. For either one of these solutions, the following holds
\begin{align*}
	\NS(x-t,x+t)-\cR(B_p,\NS,B_q)(x-t,x+t) 
		 \sim \frac{79}{3780} t^6 x^{-5} +\O(x^{-7}),
\end{align*}
which menas that is asymptotically greater than 0:
\bes
	\NS-\cR(B_p,\NS,B_q) \succ0.
\ees
The necessary condition for the inequality $\NS>\cR(B_p,P,B_q)$ is fulfilled.
If $\NS$ is $(B_p,B_q)$-sub-stabilizable for $(p,q)=(\pm\sqrt7,2\mp \frac{\sqrt7}2)$, then those parameters are the best possible. Numerical experiments and plotting the graph of a difference $\NS-\cR(B_p,\NS,B_q)$ on line $(s,1-s)$, $s\in[0,1]$, indicate that $\NS$ should be $(B_p,B_q)$-sub-stabilizable for these $p$ and $q$.

Similar procedure for the second Seiffert mean $T$, i.e.\ equating two coefficients with zero, does not give the difference that is always greater than 0. If only one coefficient is equated with zero, or equivalently $p=5-2q$, then the next coefficient is $a_2=\frac1{90}(5q^2-25q+22)$.
The problem of finding best parameters reduces to finding $q$ which minimizes the expression $a_2$, with $\lvert 10q-25 \rvert \ge \sqrt{185}$, and that inequality $T-\cR(B_p,T,B_q) >0$ still holds.
\end{subsection}
\end{section}


\begin{section}{Conclusion}\label{section-conclusion}

In this paper we have derived the complete asymptotic expansions of the resultant mean-map and consequently obtained the asymptotic expansion of stable (balanced) mean. 
Furthermore, we obtained the asymptotic expansion of stabilizable and stabilized means.
All the asymptotic expansions were given in a form of recursive relations for their coefficients. Besides the form of Theorems presented in Section 3, 
given asymptotic expansions may be used to obtain any unknown of three means involved in stabilizability problem.
Significance of the main results was is shown by examples of various types. 

Based on the asymptotic equalities from Section 4 
we may state the following.
\begin{conj}
\begin{enumerate}
	\item If mean $N$ is simultaneously $(K,M)$ and $(M,K)$-stabilizable, then $N=K=M$.
	\item If mean $M$ is simultaneously $(K,N)$ and $(N,K)$-stabilized, then either $M=K=N$ or $M=G=K\otimes N$.
	\item If mean $M$ is simultaneously $(K,N)$-stabilizable and $(K,N)$-stabilized, then $M=K=N$.
\end{enumerate}
\end{conj}
Notice we have proved the asymptotic equality between $K$ and $M$ in 1., and between $K$ and $N$ in 2.\ and 3.

Regarding questions from the Introduction, based on reasoning from Section 5, 
we may state the following.
\begin{enumerate}
\item All pairs $(p,q)$ such that Gini means are stable are 
	$\{(0,q),(p,0),(p,-p)\}.$
	All pairs $(p,q)$ such that Stolarsky means are stable are 
		$\{(2q,q),(p,2p),(p,-p)\}.$ 
	In addition, all parameters $r$ such that generalized logarithmic means are stable are 
		$\{ -2,-\frac12,1\}.$
\item The ﬁrst Seiffert mean $P$ is not stabilizable as well as the second Seiffert mean $T$ and Neuman-S\'andor mean $\NS$.
\item If $P$ is $(B_p,B_q)$-sub-stabilizable for 
	$(p,q)=(\pm\sqrt2,1\mp \frac{\sqrt2}2),$
	 then those parameters are the best possible.
\item If $\NS$ is $(B_p,B_q)$-sub-stabilizable for 
	$(p,q)=(\pm\sqrt7,2\mp \frac{\sqrt7}2),$
	then those parameters are the best possible.
\end{enumerate}

Methods presented in this paper can be used to obtain valuable information regarding means involved in other similar problems defined trough functional equations, especially when the explicit solution is not easy to find.

\end{section}


\end{document}